\documentclass{amsart}
\usepackage[latin1]{inputenc}
\usepackage{amssymb}

\newtheorem{thm}{Theorem}
\newtheorem{cor}[thm]{Corollary}
\newtheorem{lem}[thm]{Lemma}
\newtheorem{prop}[thm]{Proposition}
\theoremstyle{definition}
\newtheorem{defn}[thm]{Definition}
\theoremstyle{remark}

\numberwithin{equation}{section}

\newcommand{\To}{\longrightarrow}

\newcommand{\vertiii}[1]{{\left\vert\kern-0.25ex\left\vert\kern-0.25ex\left\vert #1 
		\right\vert\kern-0.25ex\right\vert\kern-0.25ex\right\vert}}

\begin{document}
\setcounter{tocdepth}{1}


\title[]{Spaces $C(K)$ with an equivalent URED norm}
\author{Antonio Avil\'{e}s}
\address{Departamento de Matem\'{a}ticas, Universidad de Murcia, Campus de Espinardo, 30100 Murcia, Spain.} \email{avileslo@um.es}

\author{Stanimir Troyanski}
\address{Institute of Mathematics and Informatics, Bulgarian Academy of Science, bl. 8, acad. G. Bonchev str. 1113 Sofia, Bulgaria and Departamento de Matem\'{a}ticas, Universidad de Murcia, Campus de Espinardo, 30100 Murcia, Spain.} \email{stroya@um.es}

\thanks{Authors supported by project MTM2017-86182-P (Government of Spain, AEI/ERDF-FEDER, EU). First author also supported by project 20797/PI/18 by Fundaci\'{o}n S\'{e}neca, ACyT Regi\'{o}n de Murcia, and second author also supported by the Bulgarian National Scientific Fund, Grant KP06H22/4, 04.12.2018.}


\begin{abstract}
We prove that a Banach space of continuous functions $C(K)$ has a renorming that is uniformly rotund in every direction (URED) if and only if the compact space $K$ supports a strictly positive measure.
\end{abstract}

\maketitle

\section{Introduction}

Renorming theory is a subfield of Banach space theory that deals with the construction of norms with good properties preserving the underlying topological structure of a given \emph{bad} norm. A typical such good property is rotundness. A norm is rotund (or strictly convex) if its spheres contain no nontrivial segments. In other words, if $x\neq y$ are vectors with $\|x\|=\|y\| = 1$ then $\|tx + (1-t)y\| < 1$ whenever $0<t<1$. If we look at the Banach space $C(K)$ of continuous functions on a compact space $K$, its norm $\|f\|_\infty = \max\{|f(x)| : x\in K\}$ is very far from being rotund. If two functions attain their norms at the same point with same sign and value, then the segment that joins them lies in a sphere. So the natural renorming problem here is: In spite of the lack of rotundness of the infinity norm, can we find an equivalent rotund norm on $C(K)$? This was first studied by Dashiell and Lindenstrauss \cite{DasLin} who proved that, although this is false for an arbitrary $K$, there are many large classes of compact spaces $K$ for which $C(K)$ has a rotund renorming. A remarkable later contribution is that of Haydon, who completely solved the problem when $K$ is a tree \cite{Haydon} and provided further negative examples in \cite{Haydon1}. One of the cases when $C(K)$ has an equivalent rotund norm is when $K$ admits a strictly positive measure. That is, when there is a measure $\mu$ defined on the Borel $\sigma$-algebra of $K$ such that $\mu(G)>0$ for all nonempty open $G\subset K$. The construction of a rotund norm $\vertiii{ \cdot}$ is not difficult in this case, we take the $\ell_2$-sum of the infinity norm and the norm in $L_2(\mu)$:
$$ \vertiii{f}  = \sqrt{ \|f\|^2_\infty + \int_K f^2 d\mu }. $$
A more minute analysis of this norm shows that is not only rotund, but it has a stronger property: it is \emph{uniformly rotund in every direction} (URED). This is one of the several standard strengthnings of rotundness in which a certain uniformity is required on the relation between the distance between points in the sphere and the distance of midpoints to the sphere. We recall the definition:

\begin{defn}\label{defURED}
	A norm on a space $X$ is said to be URED if whenever $(x_n)_{n}$ and $(y_n)_n$ are two sequence of vectors such that
	\begin{enumerate}
		\item $\|x_n\| = \| y_n\| = 1$ for all $n$,
		\item $\lim_n \left\|\frac{x_n+y_n}{2}\right\|=1$,
		\item There exists $z\in X$ and scalars $r_n$ such that $x_n-y_n = r_n z$ for all $n$,
	\end{enumerate}
then $\lim_n \|x_n -y_n\| = 0$.
\end{defn}

When condition (3) is removed, we would get uniform rotundity (UR), so this explains the name. We refer to the monograph \cite{DGZ} for further information. The main result of this paper is a converse to the above. This is basically the only way to get a URED renorming on a space of continuous functions:

\begin{thm}\label{maintheorem}
	For a compact space $K$ the following are equivalent:
	\begin{enumerate}
		\item $C(K)$ admits an equivalent URED norm.
		\item $K$ supports a strictly positive measure.
	\end{enumerate}
\end{thm}

This can be viewed as an improvement of a result by Rycht\'{a}\v{r} \cite{Rychtar}. He defines a norm on a space $X$ to be pointwise uniformly rotund (p-UR) if there is a weak$^\ast$ dense subspace $F$ of $X^\ast$ such that, whenever conditions (1) and (2) of Definition~\ref{defURED} are satisfied, we conclude that $\lim_n f(x_n-y_n) = 0$ for all $f\in F$. For every $z\in X$, $z\neq 0$ we can take $f\in F$ with $f(z)\neq 0$, so we get that every $p$-UR norm is URED. Rycht\'{a}\v{r}'s theorem \cite{Rychtar} states that $C(K)$ admits an equivalent $p$-UR norm if and only if $K$ supports a strictly positive measure. The norm $\vertiii{\cdot}$ defined above is in fact p-UR. Even if Theorem~\ref{maintheorem} improves Rycht\'{a}\v{r}'s result, the techniques used in the proofs are completely different. It follows that $C(K)$ has a URED renorming if and only if it has a $p$-UR renorming. This is not true for general Banach spaces. The space $L_1(\mu)$ has an equivalent URED norm, cf. \cite{KT,Kutzarova,DKT}, and \cite[Remark 8]{LPT} for a quantive version of this fact. On the other hand, if $L_1(\mu)$ is nonseparable and $\mu$ is nonatomic, then it has no equivalent $p$-UR norm \cite{DKT,Rychtar}. 

Another remark is that Theorem~\ref{maintheorem} is a well known fact when $K$ is scattered. Or more generally, when $K$ has a dense set of isolated points. In that case, having a strictly positive measure is equivalent to the fact that this set of isolated points is countable. If it is uncountable, then the functions supported on the isolated points generate a copy of $c_0(\Gamma)$ with uncountable $\Gamma$, which fails to have a URED renorming \cite{DJS}, cf. also \cite[Proposition IV.6.4]{DGZ}. Remember that a scattered space can have a countable set of isolated points but uncountable height, or can even be \emph{thin and tall} \cite{JW}.

 The implication $(2)\Rightarrow (1)$ is easy and known. For the sake of completeness, we can quickly provide a proof. After the comments above, it would remain to check that the norm $\vertiii{\cdot}$ is $p$-UR. Suppose that $\mu$ is a strictly positive measure, and take the subspace $F\subset C(K)^\ast$ of all functionals of the form $\hat{f}(g) = \int_K f g d\mu$ with $f\in L_2(\mu)$. Since $\mu$ is strictly positive, $F$ separates points, so it is weak$^\ast$ dense. If $(x_n)$ and $(y_n)$ are as in (1) and (2) in Definition~\ref{defURED}, using first Cauchy-Schwarz inequality, second the fact that the square of the norm is a convex function, and then the definition of $\vertiii{\cdot}$,

\begin{eqnarray*}
	\hat{f}(x_n-y_n)^2 = \left(\int_K f (x_n-y_n) d\mu\right)^2 \leq \left(\int_K f^2 d\mu\right) \cdot \left(\int_K (x_n-y_n)^2 d\mu\right)\\
	\leq \left(\int_K f^2 d\mu\right) \cdot \left(2(\|x_n\|_\infty^2 + \|y_n\|_\infty^2) - \|x_n+y_n\|_\infty^2 + \int (x_n-y_n)^2d\mu \right) \\
	= \left(\int_K f^2 d\mu\right) \cdot\left(2(\vertiii{x_n}^2 + \vertiii{y_n}^2) - \vertiii{x_n+y_n}^2\right) \To 0.
\end{eqnarray*}

The rest of the paper is entirely devoted to the proof of the implication  $(1)\Rightarrow (2)$ of Theorem~\ref{maintheorem}, passing through a sequence of auxiliary results. This will require some variations on the martingale characterization of URED renorming by the second author \cite{T1,T2}, using martingales that are not necessarily Walsh-Paley and taking advantage of the algebraic structre of $C(K)$. From this, we will not get an explicit strictly positive measure. Instead, we will make use of a variation, due to Galvin and Prikry \cite{GalPri}, of Kelley's characterisation of spaces with strictly positive measure through a countable decomposition into families of open sets with positive intersection numbers.

\section{Martingales in URED spaces}

 As mentioned in the introduction, an important part of the argument is a variation on the second auhor's martingale characterization of URED renorming. The is exposed in the book of Deville, Godefroy and Zizler \cite{DGZ}, that we will follow closely. Let $(\Omega,\Sigma,p)$ be a probability space. We start by recalling the notion of discrete martingale in a normed space. A general auxiliary reference for martingales with values in Banach spaces may be Stromberg's book \cite{Stromberg}. A partition of $\Omega$ is a finite pairwise disjoint family of elements of $\Sigma$ whose union is $\Omega$.  Let $(\mathcal{E}_n)_{n\geq 0}$ be a sequence of partitions of $\Omega$ consecutively finer (every element of $\mathcal{E}_{n+1}$ is a subset of an element of $\mathcal{E}_n$). We denote by $\mathcal{A}_n$ the algebra generated by the partition $\mathcal{E}_n$, i.e. $\mathcal{A}_n$ consists of all unions of elements of $\mathcal{E}_n$. Let $X$ be a linear space, $L_n:\Omega\to X$ be an $\mathcal{A}_n$-simple random variable, i.e. a random variable that is constant on every set from the partition $\mathcal{E}_n$, such that 

$$ \forall E\in \mathcal{A}_{n-1} \ \ \int_E L_n dp = 0.$$ 

Finally, let $M_n = \sum_{j=0}^n L_j$. Such a (finite or infinite) sequence of random variables $M_1, M_2,\ldots$ is called a discrete martingale, and for us just a martingale. In the language of conditional expectations, what we are saying is that $\mathbb{E}(M_n|\mathcal{A}_{n-1}) = M_{n-1}$. The increments of the martingales are denoted as usual $dM_n = L_n = M_n - M_{n-1}$ for $n\geq 1$, $dM_0=L_0$.

\begin{lem}\label{normsquareincreases}
If $X$ is a normed space, then, for every set $E\in \mathcal{A}_{n-1}$ we have $$\int_E \|M_n\|^2 \geq \int_E \|M_{n-1}\|^2.$$
\end{lem}

\begin{proof}
	This is an elementary fact, but we state it and prove it as a lemma for better reference. We can suppose that $E\in\mathcal{E}_{n-1}$. So we can write $E=\bigcup_1^k E_i$ as a partition with $E_i\in \mathcal{E}_n$. We know that $M_n$ is constant equal to a vector $x\in X$ on $E$, while $L_n$ is constant to a vector $y_i\in X$ one each $E_i$. The desired inequality is 
	$$\sum_{i=1}^k \|x+y_i\|^2\cdot p(E_i) \geq \|x\|^2 \cdot p(E). $$
	Notice that $\sum_{i=1}^k y_i p(E_i) = \int_E L_n = 0$. So we just use the fact that the square of the norm is a convex function.
\end{proof}

A $k$-Walsh-Paley pair is a pair of sets $E^+,E^-\in \mathcal{E}_k$ with the same probability and such that $E^+ \cup E^- \in \mathcal{E}_{k-1}$. The set $\Omega_k$ is the union of all $k$-Walsh-Paley pairs. A Walsh-Paley martingale is one where $\Omega=\Omega_k$ and $|\mathcal{E}_k| = 2^k$ for all $k$.

Let $H$ be a homogeneous subset of a normed space $X$. That is, $\lambda x \in H$ whenever $x\in H$ and $\lambda\in\mathbb{R}$. For $k\in\mathbb{N}$, one defines an index
$$\alpha_k(H) = \inf \left\{ \sup_n \left(\mathbb{E}\|M_n\|^2\right)^{1/2} : (M_n) \in \Xi_k\right\},$$
where $\Xi_k$ is the set of all Walsh-Paley martingales for which there exist at least $k$ many different integers $n$ and measurable sets $D_n\in \Sigma$ such that $dM_{n}(D_n)\subset H$ and 
$\int_{D_n}\|dM_n\|^2 dp \geq 1$.

Geometrically speaking, $\alpha_k(H)$ measures how fast a dyadic tree must grow when it has many large branches in $H$. Clearly $(\alpha_k(H))_{k\geq 0}$ is a nondecreasing sequence. Define also
$$\alpha(H) = \sup_k \alpha_k(H).$$

This is the aforementioned characterization:

\begin{thm}[\cite{DGZ} Theorem IV.6.1] \label{Troyanskischaracterization} 
	A normed space $X$ admits an equivalent URED norm if and only if for every $t>0$ there exists a sequence of homogeneous sets $(X_{m,t})_{m\geq 1}$ such that $X=\bigcup_m X_{m,t}$ and $\inf_m \alpha(X_{m,t}) \geq t$.
\end{thm}

We will need a variation of this result with similar invariants $a_k(H)$ and $a(H)$ instead of $\alpha_k(H)$ and $\alpha(H)$, where martingales will not necessarily be Walsh-Paley, though we will restrict to Walsh-Paley pairs.

$$a_k(H) = \inf \left\{ \sup_n \left(\mathbb{E}\|M_n\|^2\right)^{1/2} : (M_n) \in \tilde{\Xi}_k\right\},$$
where $\tilde{\Xi}_k$ is the set of all martingales for which there exist at least $k$ many different integers $n$ and measurable sets $D_n\in \Sigma$ such that $dM_{n}(D_n)\subset H$ and 
$\int_{D_n\cap \Omega_n}\|dM_n\|^2 dp \geq 1$. Again, this is a nondecreasing sequence of indices and we define
$$a(H) = \sup_k a_k(H).$$
Clearly, $a_k(H) \leq \alpha_k(H)$ and $a(H)\leq \alpha(H)$. It is also clear that if $D\subset H$ then $\beta(H)\leq\beta(D)$ for $\beta=\alpha_k,\alpha,a_k,a$.

\begin{thm}\label{UREDa}
	A normed space $X$ admits an equivalent URED norm if and only if for every $t>0$ there exists a sequence of homogeneous sets $(X_{m,t})_{m\geq 1}$ such that $X=\bigcup_m X_{m,t}$ and $\inf_m a(X_{m,t}) \geq t$.
\end{thm}

\begin{proof}
We will follow the proof of \cite[Theorem IV.6.1]{DGZ}, making changes where necessary. Since $a\leq \alpha$, the implication that does not trivially follow from Theorem~\ref{Troyanskischaracterization} is that if $X$ has a URED norm, then we have sets $X_{m,t}$ as above. We will include several lemmas inside the proof of this theorem. The first one is a useful standard characterization of URED norms:

\begin{lem}\label{UREDsquares}
	A norm in a space $X$ is URED if and only if whenever we have two sequences of vectors $(u_n)$ and $(v_n)$ such that
	\begin{enumerate}
		\item $\lim_n 2\|u_n\|^2+2\|v_n\|^2 - \|u_n+v_n\|^2 = 0$,
		\item $(u_n)$ is bounded,
		\item there exists a vector $z$ and scalars $r_n$ with $u_n-v_n = r_n z$,
	\end{enumerate} 
then $\lim_n r_n = 0$.
\end{lem}

\begin{proof}
	This is \cite[Proposition II.6.2]{DGZ}.
\end{proof}
For $t>0$ and $i\in\mathbb{N}$, define $U_i(t)$ as the set of al $x\in X$ such that
$$\inf\left\{ \frac{\|x+y\|^2 + \|x-y\|^2}{2\|y\|^2} : y\in X, \  \|y\| \leq t \|x\|  \right\} \geq 1+i^{-1}$$

\begin{lem} For every $t$, we have $X = \bigcup_{i=1}^\infty U_i(t)$.
\end{lem}

\begin{proof}
	Notice that we always have
	$$\frac{\|x+y\|^2 + \|x-y\|^2}{2\|y\|^2} \geq 1.$$
	This can be deduced from Lemma~\ref{normsquareincreases}, applied to a martingale where $L_0$ constant equal to $y$, and $L_1$ is equal to $x$ and $-x$ in two sets of equal probability. So if the lemma was false, there would exist $x\in X$ and a sequence $y_1,y_2,\ldots\in X$ such that $\| y_i\| \leq t\|x\|$ for all $i$ and
	$$\lim_i \frac{\|x+y_i\|^2 + \|x-y_i\|^2}{2\|y_i\|^2} = 1,$$
	$$\text{so } \lim_i \frac{\|x+y_i\|^2 + \|x-y_i\|^2}{2\|y_i\|^2} - 1 = 0.$$
	 Since $\|y_i\|\leq t\|x\|$, we can multiply by $4\|y_i\|^2$ and get
	 $$ \lim_i 2\|x+y_i\|^2 + 2\|x-y_i\|^2 - \| 2y_i\|^2 = 0.$$
	 We can apply Lemma~\ref{UREDsquares} for $u_i=x+y_i$, $v_i=y_i-x$, $z=x$ and $r_i=2$ for all $i$. This is a contradiction.
\end{proof}

It will be enough to prove that $a(U_i(t))\geq t/2$. In fact, we will show that
$$ (\star) \ \ k\geq it^2 \Rightarrow a_k(U_i(t)) \geq t/2$$

\begin{lem}\label{lemmafromDGZ}
Let $f, g$ be simple $X$-valued random variables on $(\Omega,\Sigma,p)$. If $\mathbb{E}\|f\|^2 \leq 1$ and
$$\int_{g^{-1}(U_i(t))}\|g\|^2 \geq 2 t^{-2},$$
then
$$\mathbb{E}(\|f+g\|^2 +\|f-g\|^2) \geq 2 \mathbb{E}\|f\|^2 + t^{-2}i^{-1}. $$
\end{lem}

\begin{proof} This is exactly \cite[Lemma IV.6.2]{DGZ}.
\end{proof}

\begin{lem}\label{furtherlemma}
	Let $(M_n)$ be a martingale such that $\sup_n \mathbb{E}\|M_n\|^2 \leq 1$. Fix $n\in\mathbb{N}$ such that
	$$\int_{(dM_n)^{-1}(U_i(t))\cap \Omega_n} \|dM_n\|^2 dp \geq 2 t^{-2}.$$
	Then $\mathbb{E}\|M_n\|^2\geq \mathbb{E}\|M_{n-1}\|^2 + t^{-2}i^{-1}$.
\end{lem}

\begin{proof}
Set $E = (dM_n)^{-1}(U_i(t))\cap \Omega_n$. Since $E\subset \Omega_n$, we can write $E = \bigcup_{j=1}^s(E_j \cup E_j ^-)$, where $E_j^+ \cup E_j^-$ are $n$-Walsh-Paley pairs, so  $E\in\mathcal{A}_{n-1}$, $p(E_j^+) = p(E_j^-)$ and $dM_n(E_j^+) = - dM_n(E_j^-)$. Hence,
$$\int_E\|M_n\|^2 dp = \int_E\|M_{n-1}+dM_n\|^2 dp = \int_{E^+} (\|M_{n-1}+dM_n\|^2 + \|M_{n-1}-dM_n\|^2)dp,$$
where $E^+ = \bigcup_{j=1}^s E_j^+$. Since $dM_n(E^+) \subset dM_n(E) \subset U_i(t)$, we can apply Lemma~\ref{lemmafromDGZ} for $f=1_{E^+}\cdot M_{n-1}$ and $g=1_{E^+}\cdot dM_{n-1}$, and we obtain
\begin{eqnarray*}
\int_{E^+}(\|M_{n-1}+dM_n\|^2 + \|M_{n-1}-dM_n\|^2)dp & \geq & 2 \int_{E^+}\|M_{n-1}\|^2 dp + t^{-2}i^{-1}\\
& = &  \int_{E}\|M_{n-1}\|^2 dp + t^{-2}i^{-1}
\end{eqnarray*}
So we conclude that
$$ \int_E\|M_n\|^2 dp \geq \int_{E}\|M_{n-1}\|^2 dp + t^{-2}i^{-1}.$$
On the other hand, $\Omega\setminus E\in \mathcal{A}_{n-1}$, so by Lemma~\ref{normsquareincreases},
$$\int_{\Omega\setminus E} \|M_n\|^2 dp \geq \int_{\Omega\setminus E} \|M_{n-1}\|^2 dp$$
The last two inequalities together prove the lemma. 
\end{proof}

To finish the proof of Theorem~\ref{UREDa} it remains to prove the inequality $(\star)$. This just imitates \cite[Lemma 6.3]{DGZ}. The first observation is that, since the scalar multiple of a martingale is a martingale, taking $\tilde{M}_n = 2t^{-1}M_n$ the definition of $a_k(H)$ can be rewritten as
\begin{eqnarray*}
	a_k(H) = \inf\{ 2^{-1}t\sup_n \left(\mathbb{E}\|\tilde{M}_n\|^2\right)^{1/2} : (\tilde{M}_n) \text{ is a martingale}\\
	 \text{and } \left|\{n : \int_{d\tilde{M}_n^{-1}(H)\cap \Omega_n}\|d\tilde{M}_n\|^2 dp \geq 4t^{-2}\}\right|\geq k \},
\end{eqnarray*}
So, if $a_k(U_i(t))<t/2$ there should exist a martingale $(\tilde{M}_n)$ such that $\sup_n \mathbb{E}\|\tilde{M}_n\|^2 <1$ while
 $$\left|\{n : \int_{d\tilde{M}_n^{-1}(H)\cap \Omega_n}\|d\tilde{M}_n\|^2 dp \geq 4t^{-2}\}\right|\geq k.$$
 
 But if $(\tilde{M}_n)$ is such a martingale and 
 $$\int_{d\tilde{M}_n^{-1}(H)\cap \Omega_n}\|d\tilde{M}_n\|^2 dp \geq 4t^{-2},$$
 then by Lemma~\ref{furtherlemma}, we have $\mathbb{E}\|\tilde{M}_n\|^2\geq \mathbb{E}\|\tilde{M}_{n-1}\|^2 + t^{-2}i^{-1}$. Taking into account Lemma~\ref{normsquareincreases} and that $\sup_n \mathbb{E}\|\tilde{M}_n\|^2 <1$, it follows that there are less than $it^2$ many such numbers $n$. We were assuming that $k\geq it^2$ so we get a contradiction.
\end{proof}

\section{Weak intersection numbers}

\begin{defn}
	Given a finite family of sets $\mathcal{D}$, $l(\mathcal{D})$ will be the least cardinality $k\in \mathbb{N}$ such that for every $\mathcal{A}\subset \mathcal{D}$ with $|\mathcal{A}| > k$ we have $\bigcap \mathcal{A} = \emptyset$.
\end{defn}

\begin{defn}
	Given a family of sets $\mathcal{B}$, we define
	$$win(\mathcal{B}) = \inf \left\{ \frac{l(\mathcal{D})}{|\mathcal{D}|} : \mathcal{D} \subset \mathcal{B} \text{ is a nonempty finite subfamily} \right\}$$
\end{defn}

This index is called \emph{the weak intersection number of} $\mathcal{B}$ by Galvin and Prikry \cite{GalPri}. It is a variation of Kelley's intersection number, that can be alternatively used in the celebrated Kelley's characterization of compact spaces supporting a strictly positive measure , cf. \cite[Theorem 2]{GalPri} and \cite[Corollary 2.7]{TodorcevicCC}:

\begin{thm}\label{Kelleyschar}
	For a a compact Hausdorff space $K$ the following are equivalent:
	\begin{enumerate}
		\item $K$ supports a strictly positive measure.
		\item The family $\mathcal{G}$ of all nonempty open subsets of $K$ admits a countable decomposition $\mathcal{G} = \bigcup_{n=1}^\infty \mathcal{G}_n$ such that $win(\mathcal{G}_n)>0$ for all $n$.
	\end{enumerate} 
\end{thm}  

For technical reasons, we will consider a slight variation:

\begin{defn}
	Given a family of sets $\mathcal{B}$, we define
	$$wi\tilde{n}(\mathcal{B}) = \inf \left\{ \frac{l(\mathcal{D})}{|\mathcal{D}|-1} : \mathcal{D} \subset \mathcal{B} \text{ is a finite subfamily with } |\mathcal{D}|>1  \right\}$$
\end{defn}

\begin{thm}\label{Kelleyscharbeta}
	For a a compact Hausdorff space $K$ the following are equivalent:
	\begin{enumerate}
		\item $K$ supports a strictly positive measure.
		\item The family $\mathcal{G}$ of all nonempty open subsets of $K$ admits a countable decomposition $\mathcal{G} = \bigcup_{n=1}^\infty \mathcal{G}_n$ such that $wi\tilde{n}(\mathcal{G}_n)>0$ for all $n$.
	\end{enumerate} 
\end{thm}  

\begin{proof}
	It is enough to check that, given a family $\mathcal{B}$ of nonempty sets, $wi\tilde{n}(\mathcal{B})>0$ if and only if $win(\mathcal{B})>0$. For $k=1,2,3,\ldots$, consider
	$$\gamma_k(\mathcal{B}) = \inf \left\{ l(\mathcal{D}) : \mathcal{D} \subset \mathcal{B},\  |\mathcal{D}|=k  \right\}.$$
	Notice that $\gamma_k(\mathcal{B})>0$ for all $k$, and the infimum of a sequence of positive numbers is positive if and only if the lower limit of that sequence is positive. Therefore,
	\begin{eqnarray*}
	win(\mathcal{B}) =
	\inf_{k\geq 1}  \frac{\gamma_k(\mathcal{B})}{k}>0
	& \iff & \liminf_{k\geq 1}\frac{\gamma_k(\mathcal{B})}{k}>0 \\
	& \iff & \liminf_{k\geq 2}\frac{\gamma_k(\mathcal{B})}{k-1}\cdot\frac{k-1}{k}= \liminf_{k\geq 2}\frac{\gamma_k(\mathcal{B})}{k-1}>0\\
	& \iff & wi\tilde{n}(\mathcal{B}) =
	\inf_{k\geq 2}  \frac{\gamma_k(\mathcal{B})}{k-1}>0.
	\end{eqnarray*}
\end{proof}

Given a nonempty open set $G\subset K$, we choose, using Urysohn's lemma, a continuous function $f_G\in C(K)$ such that $\|f_G\|_\infty = 1$ and $f_G|_{K\setminus G} = 0$. Given a family $\mathcal{G}$ of open sets, the corresponding family of functions will be written as
$$\mathcal{F}_{\mathcal{G}} = \{ f_G : G\in \mathcal{G}\}.$$

\begin{lem}\label{keylemma}
	Let $\mathcal{G} = \{G_1,\ldots,G_m\}$ be a finite family of nonempty open subsets of $K$. For every $g\in C(K)$ there exists a $C(K)$-valued martingale $(N_0,N_1,N_2)$ on a probability space $(\Omega,\Sigma,p)$ such that
	\begin{enumerate}
		\item $\Omega_2 = \Omega$, 
		
		and for all $\omega\in \Omega$:
		\item $N_0(\omega) = g$,
		\item either $dN_2(\omega)\in \mathcal{F}_\mathcal{G}$ or $-dN_2(\omega)\in \mathcal{F}_\mathcal{G}$,
		\item\label{keylemmaformula} $\|N_2(\omega)\|_\infty \leq \max(\|g\|_\infty,1) + \frac{l(\mathcal{G})}{m-1}$.
	\end{enumerate}
\end{lem}

\begin{proof}
	Consider $(E_k^{\pm})_{k=1}^m$ a sequence of $2m$ many pairwise disjoint measurable sets, whose union is $\Omega$ and of equal measure $p(E_k^\pm) = \frac{1}{2m}$. Consider $g_1,g_0 \in C(K)$ defined by
	$$ g_1(x) = \begin{cases} 
	-1 & \text{ if } g(x)\leq -1,\\
	g(x) & \text{ if } -1<g(x)<-1, \\
	1 & \text{ if } g(x)\geq 1,
	 \end{cases}$$
	 $$g_0(x) = g(x) - g_1(x) = \begin{cases} 
	 g(x)+1 & \text{ if } g(x)\leq -1,\\
	 0 & \text{ if } -1<g(x)<-1, \\
	 g(x)-1 & \text{ if } g(x)\geq 1.
	 \end{cases}$$
	Notice that $\|g_0\|_\infty  = (\|g\|_\infty-1)^+$. Let us now define the martingale $N_n$ and associated $L_n = dN_n$, $n=0,1,\ldots$. We declare $N_0 = g = g_0 + g_1$, $L_0=N_0$, 
	$$L_1(\omega) = g_1 \cdot \left(\frac{h_k}{m-1} - f_{G_k}\right), \text{ for } \omega\in E_k^-\cup E_k^+$$
	where $$h = \sum_{j=1}^mf_{G_j}, \ \ h_k = h - f_{G_k}.$$
	From the definition of  $l(\mathcal{G})$, it follows that $\|h_k\|_\infty\leq l(\mathcal{G})$ for all $k$.
	Notice that
	$$\mathbb{E}(L_1) = \frac{g_1}{m}\sum_{k=1}^m \left(\frac{h_k}{m-1} - f_{G_k}\right) = \frac{g_1}{m}\left(\sum_{k=1}^m \frac{\sum_{j\neq k}f_{G_j}}{m-1} - \sum_{k=1}^m f_{G_k}\right) = 0.$$
	So we can define $N_1 = N_0 + L_1$, and finally $L_2(\omega) = \pm f_{G_k}$ when $\omega\in E_k^\pm$ and $N_2 = N_1+L_2$. The nontrivial part that remains to be proven is statement (\ref{keylemmaformula}). For $\omega\in E_k^\pm$, we have that
	\begin{eqnarray*}
		N_2(\omega) = N_0(\omega) + L_1(\omega) + L_2(\omega) = g  + {g_1}\left(\frac{h_k}{m-1} - f_{G_k}\right) \pm f_{G_k}\\
		= g - g_1 f_{G_k} + \frac{g_1 h_k}{m-1} \pm f_{G_k} = g_0 + g_1(1-f_{G_k}) \pm f_{G_k} + \frac{g_1 h_k}{m-1}
	\end{eqnarray*}
We already noticed that $\|h_k\|\leq l(\mathcal{G})$ and it is obvious from the definition of $g_1$ that $\|g_1\|\leq 1$. So the last summand is bounded by $\frac{l(\mathcal{G})}{m-1}$. We also computed that $\|g_0\|_\infty = (\|g\|_\infty-1)^+$, so in order to prove (\ref{keylemmaformula}), it is enough to  show that the central summands satisfy
$$\|g_1(1-f_{G_k}) \pm f_{G_k}\|_\infty \leq 1.$$
But this is obvious, because $0\leq f_{G_k}\leq 1$ and $\|g_1\|_\infty\leq 1$, so the above function takes as value, on every point $x\in K$, a convex combination of $g_1(x)$ and $\pm 1$.
\end{proof}

\begin{prop}
	Let $\mathcal{G}_1,\ldots,\mathcal{G}_q$ be a finite sequence of finite families of nonempty open subsets of $K$. Then there exists a $C(K)$-valued martingale $(M_0,\ldots,M_{2q})$ on a probability space $(\Omega,\Sigma,p)$ such that
	\begin{enumerate}
		\item $\Omega_{2r} = \Omega$ for $r=1,\ldots,q$,
		\item For all $\omega\in \Omega$ and $r=1,\ldots,q$, either $dM_{2r}(\omega)\in \mathcal{F}_{\mathcal{G}_{2r}}$ or $-dM_{2r}(\omega)\in \mathcal{F}_{\mathcal{G}_{2r}}$,
		\item For all $\omega\in\Omega$,
		$$\|M_{2q}(\omega)\|_\infty \leq 1 + \sum_{r=1}^q \frac{l(\mathcal{G}_r)}{|\mathcal{G}_r|-1}.$$
	\end{enumerate}
\end{prop}

\begin{proof}
	We construct $M_r$ by induction on $r$ for $r=0,\ldots,2q$ on a given nonatomic probability space. We take $M_0$ to be constant equal to an arbitrary function $g\in C(K)$ with $\|g\|=1$. Assume that $M_r$ has been constructed for $r=0,\ldots,2k$. We consider $\mathcal{E}_{2k} = \{E_1,\ldots,E_m\}$ the atoms of the algebra $\mathcal{A}_{2k}$. We know that $M_{2k}$ is constant on each $E_j$ equal to some $g_j\in C(K)$ and, for all $\omega\in\Omega$,
			$$\|M_{2k}(\omega)\|_\infty \leq 1 + \sum_{r=1}^k \frac{l(\mathcal{G}_r)}{|\mathcal{G}_r|-1}.$$
	For every $j\in\{1,\ldots,m\}$ we can apply Lemma~\ref{keylemma} to the probability space $(E_j,\Sigma|_{E_j},\frac{p}{p(E_j)})$, $g=g_j$, and the family $\mathcal{G} = \mathcal{G}_{k+1}$. This provides a martingale $(N_0^j,N_1^j,N_2^j)$ on $E_j$. Putting all these martingales together, we can define $M_{2k+1}(\omega) = N^j_1(\omega)$ and $M_{2k+2}(\omega) = N^j_2(\omega)$ whenever $\omega\in E_j$. These random variables have all the required properties.
\end{proof}

\begin{cor}\label{corabound}
	If $H\subset C(K)$ is a homogeneous set such that $\mathcal{F}_{\mathcal{G}_1} \cup\cdots\cup\mathcal{F}_{\mathcal{G}_q}\subset H$, then
		$$a_{2q}(H) \leq 1 + \sum_{r=1}^q \frac{l(\mathcal{G}_r)}{|\mathcal{G}_r|-1}.$$ 
\end{cor}

\begin{proof}
 The martingale in the previous lemma is one of the martingales that appear in the definition of $a_k(H)$ because, for even numbers $n\leq 2k$,
 $$ \int_{dM_n^{-1}(H)\cap \Omega_n}\|dM_n\|_\infty^2 dp =  \int_{\Omega}\|dM_n\|_\infty^2 dp = 1.$$
\end{proof}

\begin{cor}\label{betawiszero}
	Let $\mathcal{G}$ be a family of nonempty open sets with $wi\tilde{n}(\mathcal{G})=0$. Then $a(H)\leq 1$ for every homogenous set $H\subset C(K)$ such that $\mathcal{F}_\mathcal{G}\subset H$.
\end{cor}

\begin{proof}
	Fix $\varepsilon>0$ and we will prove that $a(H)<1+\varepsilon$. Write $\varepsilon =\sum_{r=0}^\infty \varepsilon_r$ for some numbers $\varepsilon_r>0$. Since $wi\tilde{n}(\mathcal{G})=0$ for every $r\in\mathbb{N}$ there exists a finite set $\mathcal{G}_r\subset \mathcal{G}$ such that
	$$ \frac{l(\mathcal{G}_r)}{|\mathcal{G}_r|-1} < \varepsilon_r. $$
	By Corollary~\ref{corabound}, we get that for every $q\in\mathbb{N}$,
	$$a_{2q}(H) \leq  1 + \sum_{r=1}^q \frac{l(\mathcal{G}_r)}{|\mathcal{G}_r|-1} < 1 +\sum_{r=1}^q\varepsilon_r < 1+\varepsilon.$$ 
	But $(a_k(H))_{k\geq 1}$ is a nondecreasing sequence, so $a(H) = \sup_k a_k(H) \leq 1.$
\end{proof}

We are now ready to prove the implication $(1)\Rightarrow (2)$ of  Theorem~\ref{maintheorem}. Assume that $C(K)$ has an equivalent URED norm. Fix any $t>1$ and consider the decomposition $X=\bigcup_m X_{m,t}$ provided by Theorem~\ref{UREDa}. Consider also 
$$\mathcal{G}_{m,t} = \{G \text{ nonempty open sets} : f_{G} \in X_{m,t} \}.$$
Notice that $\bigcup_m \mathcal{G}_{m,t}$ is the family of all nonempty open sets. If $K$ did not support a measure, then by Theorem~\ref{Kelleyscharbeta} there must $m$ such that $wi\tilde{n}(\mathcal{G}_{m,t}) = 0$. Since $\mathcal{F}_{\mathcal{G}_{m,t}} \subset X_{m,t}$,  Corollary~\ref{betawiszero} implies that $a(X_{m,t})\leq 1$. This contradicts the fact given by Theorem~\ref{UREDa} that $a(X_{m,t})>t$.


\begin{thebibliography}{999}

\bibitem{DasLin} F. K. Dashiell, J. Lindenstrauss,
Some examples concerning strictly convex norms on C(K) spaces.
Israel J. Math. 16 (1973), 329--342.

\bibitem{DJS} M. M. Day, R. C. James, and S. Swaminathan, Normed linear spaces that are uniformly convex in every direction, Canad. J. Math. 23 (1971), 1051--1059. 

\bibitem{DGZ} R. Deville, G. Godefroy, V. Zizler, Smoothness and renormings in Banach spaces, Pitman Monographs and Surveys in Pure and Applied Mathematics 64, Longman Scientific and Technical, Harlow, 1993.

\bibitem{DKT} S. J. Dilworth, D. Kutzarova and S. L.
Troyanski, On some uniform geometric properties in
function spaces, General topology in Banach spaces,
127--135, Nova Sci. Publ., Huntington, NY, 2001. 
 
\bibitem{GalPri} F. Galvin, K. Prikry, On Kelley's intersection numbers, Proc. Am. Math. Soc. 129 (2000), 315--323.

\bibitem{Haydon} R. G. Haydon, Trees in renorming theory, Proc. London Math. Soc. 78 (1999) 541--584. 

\bibitem{Haydon1} R. G. Haydon, Boolean rings that are Baire spaces. Serdica Math. J. 27 (2001), 91--106.

\bibitem{JW}  I. Juh\'{a}sz and W. Weiss, On thin-tall scattered spaces, Colloq. Math. 40 (1978/79), no. 1, 63--68.


\bibitem{Kutzarova} D. N. Kutzarova, On an equivalent norm in $L_1$ which is uniformly convex in every direction, Constructive Theory of Functions, Sofia 84 (1984), 507--512.

\bibitem{KT} D. Kutzarova, S. Troyanski, On equivalent lattice norms which are uniformly convex or uniformly differentiable in every  direction in Banach lattices with a weak unit, Serdica Math. J.,  9 (1983), 249--262.

\bibitem{LPT} S. Lajara,  A. Pallar\'{e}s, S. Troyanski, Estimations for the moduli of convexity and smoothness of reflexive subspaces of $L_1$, J. Funct. Anal.,  261 (2011), 3211--3225.

\bibitem{Rychtar} J. Rycht\'{a}\v{r}, Pointwise uniformly rotund norms, Proc. Am. Math. Soc. 133 (2005), 2259--2266.

\bibitem{Stromberg} K. Stromberg, Probability for analysts, Chapman
and Hall , New York , 1994.

\bibitem{TodorcevicCC} S. Todorcevic, Chain condition methods in topology, Top. Appl. 101 (2000), 45--82.

\bibitem{T1} S. Troyanski, Equivalent norms that are uniformly convex
and uniformly differentiable in every direction ,
C. R. Acad. Bulg. Sci., 32 (1979), 1461--1464.
(Russian)

\bibitem{T2} S. Troyanski, Construction of equivalent norms for certain
local characteristics with rotundity and smoothness
by means of martingales, Proc.14th Spring
Conference of Mathematics and Math. Education,
Bulg. Acad. Nauk, Sofia 1985, 129--156. (Russian)

. 


\end{thebibliography}
\end{document}